\documentclass[11pt,abstract]{scrartcl}
\usepackage[english]{babel}
\usepackage{amsmath,amssymb,amsthm}
\usepackage[final]{microtype}
\usepackage{physics}
\usepackage{graphicx}
\usepackage{xcolor}
\usepackage{tikz}
\usetikzlibrary{arrows.meta, external, calc, pgfplots.groupplots, spy, matrix}
\tikzset{external/only named=true}

\usepackage{pgfplots}
\pgfplotsset{compat=newest}

\usepackage[backend=biber,maxnames=99,hyperref,backref,bibencoding=utf8,defernumbers=true,style=numeric-comp,giveninits=true,doi=true,url=false,isbn=false]{biblatex}
\renewbibmacro{in:}{\ifentrytype{article}{}{\printtext{\bibstring{in}\intitlepunct}}}
\usepackage{csquotes}
\renewbibmacro*{doi+eprint+url}{%
	\printfield{doi}%
	\newunit\newblock%
	\iftoggle{bbx:eprint}{%
		\usebibmacro{eprint}%
	}{}%
	\newunit\newblock%
	\iffieldundef{doi}{%
		\usebibmacro{url+urldate}}%
	{}%
}
\renewbibmacro{in:}{\ifentrytype{article}{}{\printtext{\bibstring{in}\intitlepunct}}}
\DeclareFieldFormat{pages}{#1}

\usepackage[hidelinks]{hyperref}
\usepackage{cleveref}

\newcommand{\abssec}[1]{\noindent\small {\bfseries #1\quad}\ignorespaces}
\renewenvironment{abstract}{\abssec{Abstract}}{\par\vspace{1em}}

\renewcommand{\vec}[1]{{\boldsymbol{#1}}}
\newcommand{\Mesh}{{\mathcal{T}_h}}
\newcommand{\Facets}[1]{{\mathcal{F}(#1)}}
\newcommand{\RTIk}[1]{{I_{#1}^{\operatorname{RT}}}}
\newcommand{\BDMIk}[1]{{I_{#1}^{\operatorname{BDM}}}}
\newcommand{\Hdivnull}{{\vec{H}_0(\operatorname{div},\Omega)}}
\newcommand{\BRI}{{I_h^{\text{BR}}}}

\makeatletter
\let\blx@rerun@biber\relax
\makeatother

\definecolor{unianthrazit}{rgb}{0.416,0.408,0.435}
\definecolor{uniorange}{rgb}{0.929,0.431,0.0}
\definecolor{adjyellow}{rgb}{0.929,0.651,0.0}
\definecolor{adjred}{rgb}{0.898,0.0,0.086}
\definecolor{compblue}{rgb}{0.004,0.537,0.561}
\pgfplotsset{colormap={unibw}{color(0cm)=(compblue);color(1cm)=(adjyellow);color(2cm)=(uniorange);color(3cm)=(adjred)}, colormap name=unibw}

\pgfplotsset{CRuniform/.style={compblue, thick, mark=+}}
\pgfplotsset{CRanisotropic/.style={uniorange, thick, mark=+, mark options={rotate=45}}}

\pgfplotsset{CRRTuniform/.style={compblue, thick, mark=triangle}}
\pgfplotsset{CRRTanisotropic/.style={uniorange, thick, mark=triangle, mark options={rotate=60}}}

\pgfplotsset{CRBDMuniform/.style={compblue, thick, mark=square}}
\pgfplotsset{CRBDManisotropic/.style={uniorange, thick, mark=square, mark options={rotate=45}}}

\pgfplotsset{ConvergenceOrder/.style={unianthrazit, thick, dashed}}
\pgfplotsset{ConvergenceOrderReduced/.style={unianthrazit, thick, dotted}}

\theoremstyle{plain}
\newtheorem{theorem}{Theorem}
\newtheorem{lemma}[theorem]{Lemma}

\newtheorem{corollary}[theorem]{Corollary}
\theoremstyle{definition}

\title{Pressure-robust and conforming discretization of the Stokes equations on anisotropic meshes}
\author{Volker Kempf}
\date{\today}

\begin{document}	
	\maketitle
	
	\begin{abstract}
		Pressure-robust discretizations for incompressible flows have been in the focus of research for the past years.
		Many publications construct exactly divergence-free methods or use a reconstruction approach \cite{Linke2014} for existing methods like the Crouzeix--Raviart element in order to achieve pressure-robustness.
		To the best of our knowledge, except for our recent publications \cite{ApelKempfLinkeMerdon2021,ApelKempf2021}, all those articles impose a condition on the shape-regularity of the mesh, and the two mentioned papers that allow for anisotropic elements use a non-conforming velocity approximation.
		Based on the classical Bernardi--Raugel element we provide a conforming pressure-robust discretization using the reconstruction approach on anisotropic meshes.
		Numerical examples support the theory.
	\end{abstract}
	
	\section{Introduction}
		During the last years, pressure-robustness has emerged as an important property that discretizations for incompressible flow problems should possess.
		For the Stokes problem in a domain $\Omega$ that for a data function $\vec{f}\in \vec{L}^2(\Omega)$ and viscosity $\nu>0$ is given by
		\begin{subequations}\label{eq:Stokes}
			\begin{align}
				&&-\nu \Delta \vec{u} + \nabla p &= \vec{f} &&\text{in } \Omega,&&\\
				&&\nabla \cdot \vec{u} &= 0 &&\text{in } \Omega,&&
			\end{align}
		\end{subequations}
		a pressure-robust method yields velocity error estimates of the form, see \cite{Linke2014},
		\begin{equation*}
			\norm{\vec{u}-\vec{u}_h}_{1,h} \lesssim \inf_{\vec{v}_h \in \vec{X}_h} \norm{\vec{u}-\vec{v}_h}_{1,h} + h^m \abs{\vec{u}}_{m+1},
		\end{equation*}
		where $\vec{X}_h$ is the discrete velocity space and $\norm{\vec{v}}_{1,h}^2 = \sum_{T\in\Mesh} \norm{\nabla\vec{v}}^2_{0,T}$.
		Missing pressure-robustness on the other hand, e.g., in the case of the classical family of Taylor--Hood elements, leads to error estimates of the type
		\begin{equation*}
			\norm{\vec{u}-\vec{u}_h}_{1} \lesssim \inf_{\vec{v}_h \in \vec{X}_h} \norm{\vec{u}-\vec{v}_h}_1 + \frac{1}{\nu} \inf_{q_h \in Q_h} \norm{p-q_h}_0,
		\end{equation*}
		where $Q_h$ is the discrete pressure space.
		Both estimates contain the best-approximation error for the velocity in the discrete velocity space, however the advantage of the first estimate is obvious and leads to the descriptive name \emph{pressure-robust}:
		the velocity error does not depend on the pressure approximability and the viscosity of the fluid.
		
		Due to intensive research, many pressure-robust methods are known, e.g., the Scott--Vogelius element \cite{ScottVogelius1985}, $\vec{H}(\mathrm{div})$-conforming discontinuous Galerkin methods \cite{CockburnKanschatSchotzau2007,LehrenfeldSchoberl2016} or classical methods using a reconstruction approach to gain pressure-robustness \cite{Linke2014}.
		The proofs for all of these methods however rely on the assumption of shape-regularity on the mesh elements, which excludes anisotropically graded meshes for boundary layers or edge singularities, which may occur in flow problems.
		This shortcoming was treated in our publications \cite{ApelKempfLinkeMerdon2021,ApelKempf2021}, where the pressure-robust variant of the Crouzeix--Raviart method was used and we could show error estimates for anisotropic meshes in the boundary layer and edge singularity settings.
		
		Since the velocity approximation of the Crouzeix--Raviart method is non-conforming, the aim of this contribution is to present a pressure-robust and conforming method which can be used for meshes that contain anisotropic elements.
		The presented theory of this paper is contained in \cite{Kempf2022:PhD} in a more abstract setting.
	
	\section{Reconstruction approach for pressure-robustness}
		In order to achieve pressure-robustness, we employ the reconstruction approach introduced in \cite{Linke2014}.
		Consider problem \eqref{eq:Stokes} on a domain $\Omega\subset\mathbb{R}^2$ with viscosity parameter $\nu > 0$ and homogeneous Dirichlet boundary conditions. The weak form of this problem is well known: Find $(\vec{u},p)\in\vec{X}\times Q = \vec{H}_0^1(\Omega)\times L_0^2(\Omega)$ so that
		\begin{align}\label{eq:Stokesweak}
			&\nu (\nabla \vec{u}, \nabla\vec{v}) - (\nabla\cdot\vec{v},p) - (\nabla\cdot\vec{u},q) = (\vec{f},\vec{v}) &&\forall (\vec{v}, q) \in \vec{X}\times Q.
		\end{align}
		Since we later require that for the solution $(\vec{u},p)\in\vec{H}^2(\Omega)\times H^1(\Omega)$ holds, we assume that $\Omega$ is a convex polygon where this required regularity is guaranteed \cite{KelloggOsborn1976}.
		
		By using the Helmholtz--Hodge decomposition of the data $\vec{f}=\mathbb{P} \vec{f}+\nabla\phi$ into a divergence-free part $\mathbb{P}\vec{f}$ and an irrotational part $\nabla\phi$, and looking at the problem in the subspace of divergence free functions $\vec{X}^0=\{\vec{v}\in\vec{X}: (\nabla\cdot\vec{v},q)=0\ \forall q\in Q\}$
		\begin{equation}\label{eq:stokes_elliptic}
			\text{Find}\quad \vec{u}\in\vec{X}^0 \quad\text{so that}\quad \nu(\nabla \vec{u},\nabla\vec{v}) = (\vec{f},\vec{v}) = (\mathbb{P}\vec{f},\vec{v}) \quad\forall \vec{v}\in\vec{X}^0,
		\end{equation}
		we see that the velocity solution is independent of the gradient part $\nabla \phi$ of the data, see \cite{Linke2014}, as the test functions from $\vec{X}^0$ are $\vec{L}^2$-orthogonal on gradients.
		We aim to preserve this property in the discrete setting by using a reconstruction operator $I_h$, see \cite{Linke2014}, on the velocity test functions on the right hand side of the problem, so that the discrete version of \eqref{eq:Stokesweak} is given by
		\begin{align}\label{eq:Stokesdiscrete}
			&\nu a_h(\vec{u}_h,\vec{v}_h) + b_h(\vec{v}_h,p_h) + b_h(\vec{u}_h,q_h) = (\vec{f},I_h \vec{v}_h) &&\forall (\vec{v}_h, q_h) \in \vec{X}_h\times Q_h,
		\end{align}
		where $a_h(\vec{u}_h,\vec{v}_h)=(\nabla \vec{u}_h, \nabla\vec{v}_h)$ and $b_h(\vec{v}_h,p_h) = -(\nabla\cdot\vec{v}_h,p_h)$.
		Similar to \eqref{eq:stokes_elliptic} we can write this problem in the subspace of discretely divergence-free functions $\vec{X}_h^0=\{\vec{v}_h\in\vec{X}_h:b_h(\vec{v_h},q_h)=0\ \forall q_h\in Q_h\}$:
		\begin{equation}\label{eq:stokesdiscrete_elliptic}
			\text{Find}\quad \vec{u}_h\in\vec{X}_h^0 \quad\text{so that}\quad \nu a_h(\vec{u}_h,\vec{v}_h) = (\vec{f},I_h\vec{v}_h) \quad\forall \vec{v}_h\in\vec{X}_h^0.
		\end{equation}
	
		The reconstruction operator $I_h:\vec{X}_h\to \vec{H}_0(\mathrm{div},\Omega)=\{\vec{v}\in\vec{H}(\mathrm{div},\Omega): \vec{v}\cdot\vec{n}_{\partial\Omega} = 0\} $ needs to satisfy the properties
		\begin{subequations}\label{eq:as:reconstruction_operator}
			\begin{align}
				\nabla\cdot(I_h\vec{v}_h) &= \nabla \cdot \vec{v}_h &&\forall \vec{v}_h\in\vec{X}_h^0,\label{eq:as:reconstruction_operator_divergence}\\
				\norm{\vec{v}_h-I_h \vec{v}_h}_0 &\lesssim h \norm{\vec{v}_h}_{1,h} &&\forall \vec{v}_h\in\vec{X}_h. \label{eq:as:reconstruction_operator_interpolation}
			\end{align}
		\end{subequations}
		This way, the right hand side of \eqref{eq:Stokesdiscrete}, when tested with $\vec{v}_h\in\vec{X}_h^0$, satisfies
		\begin{equation*}
			(\vec{f},I_h\vec{v}_h) = (\mathbb{P}\vec{f},I_h\vec{v}_h) + (\nabla\phi,I_h\vec{v}_h) = (\mathbb{P}\vec{f},I_h\vec{v}_h).
		\end{equation*}
	
	\section{Modified Bernardi--Raugel discretization and error estimates}
		For the Bernardi--Raugel method, the velocity and pressure approximation spaces are defined by, see \cite{BernardiRaugel1985},
		\begin{align*}
			&\vec{X}_h = (\vec{P}_1(\mathcal{T}_h)\oplus \mathrm{span}\{\lambda_F^1\lambda_F^2\vec{n}_F\ \forall F\in\mathcal{F}_h\})\cap\vec{X}, \\
			&Q_h = \{q_h\in L^2(\Omega): q_h|_T \in P_0(T)\ \forall T\in\mathcal{T}_h\},
		\end{align*}
		where $\mathcal{T}_h$ is the set of mesh elements, $\mathcal{F}_h$ is the set of mesh edges, $\vec{n}_F$ the unit normal on facet $F$, and $\lambda_F^i$ the linear nodal basis functions associated with the endpoints of facet $F$.
		Thus, the velocity space is the space of continuous piecewise linear functions enriched by normal-weighted quadratic facet bubble functions and the pressure space is the space of piecewise constants.
		
		With $I_h=\operatorname{id}$ we get the standard Bernardi--Raugel method (BR), while for the pressure-robust modification we can choose $I_h$ as the lowest-order Raviart--Thomas (BR-RT) or Brezzi--Douglas--Marini (BR-BDM) interpolation operators, see \cite{LinkeMerdon2016}, which we write as $\RTIk{0}$ and $\BDMIk{1}$, respectively.
		
		\begin{lemma}\label{lem:reconstruction_operator}
			Let $\vec{X}_h$ and $Q_h$ be the Bernardi--Raugel finite element pair and let the reconstruction operator $I_h$ be defined by either $(I_h\vec{v}_h)|_T = \BDMIk{1}\vec{v}_h|_T$ or $(I_h\vec{v}_h)|_T = \RTIk{0}\vec{v}_h|_T$ for all $\vec{v}_h\in\vec{X}_h$ and $T\in\Mesh$.
			Then $I_h$ satisfies \eqref{eq:as:reconstruction_operator} independently of the mesh aspect ratio.
		\end{lemma}
		\begin{proof}
			Since $\vec{X}_h\subset \vec{C}(\overline{\Omega})\cap \vec{X}$, the operator $I_h$ maps to a subspace of $\Hdivnull$.
			Estimate \eqref{eq:as:reconstruction_operator_interpolation} is proved by summing the elementwise error estimates for the Raviart--Thomas and Brezzi--Douglas--Marini interpolation operators from \cite{AcostaApelDuranLombardi2011} and \cite{ApelKempf2020}, respectively.
			
			To show \eqref{eq:as:reconstruction_operator_divergence} we prove that the reconstruction operator preserves the discrete divergence of functions from $\vec{X}_h$, i.e.,
			\begin{equation*}
				\int_T \nabla\cdot I_h\vec{v}_h q_h\dd{\vec{x}} = \int_T \nabla\cdot\vec{v}_h q_h\dd{\vec{x}}\qquad\forall q_h \in Q_h
			\end{equation*}
			holds for all $\vec{v}_h\in\vec{X}_h$ and all $T\in\Mesh$.
			Integrating by parts we get
			\begin{equation*}
				\int_T \nabla\cdot(I_h\vec{v}_h - \vec{v}_h) q_h\dd{\vec{x}} = \int_T (\vec{v}_h - I_h\vec{v}_h)\cdot \nabla q_h\dd{\vec{x}} + \sum_{F\in\Facets{T}}\int_F (I_h\vec{v}_h - \vec{v}_h)\cdot\vec{n}_F q_h\dd{\vec{s}},
			\end{equation*}
			where $\mathcal{F}(T)$ is the set of facets of the element $T$.
			Since $q_h$ is piecewise constant it holds $\nabla q_h = \vec{0}$ and by using the definition of the operators $\BDMIk{1}$ and $\RTIk{0}$ we see that the right hand side vanishes.
		\end{proof}
	
		\begin{lemma}
			There is an operator $I_F:\vec{X}\to\vec{X}_h$ that for all $\vec{v}\in\vec{X}$ satisfies the properties
			\begin{align*}
				b_h(\vec{v},q_h) &= b_h(I_h^F\vec{v},q_h) \qquad \forall q_h\in Q_h, \\
				\norm{I_h^F\vec{v}}_{1,h} &\leq C_F \norm{\vec{v}}_{1,h},
			\end{align*}
			with a stability constant $C_F$ that is independent of the aspect ratio of the mesh and the mesh size parameter $h$.
		\end{lemma}
		\begin{proof}
			This is proofed in \cite[Theorem 1]{ApelNicaise2004} for a wide class of anisotropic two-dimensional meshes.
			In particular boundary layer adapted meshes are included in the results from the reference.
		\end{proof}
		The previous lemma provides the inf-sup stability result for the Bernardi--Raugel method in the form
		\begin{equation}\label{eq:discrete_inf_sup}
			\inf_{0\neq q_h\in Q_h} \sup_{\vec{0}\neq\vec{v}_h\in\vec{X}_h}\frac{b_h(\vec{v}_h,q_h)}{\norm{\vec{v}_h}_{1,h} \norm{q_h}_0} \geq \widetilde{\beta} >0,
		\end{equation}
		where $\widetilde{\beta}$ is the discrete inf-sup constant, as the existence of a Fortin operator is equivalent to inf-sup stability, see, e.g., \cite[Lemma 4.19]{ErnGuermond2004}.
		We have to keep in mind that the results from \cite{ApelNicaise2004} that are used in the proof are restricted to a wide class of two-dimensional meshes.
	
		The next result is a consistency estimate in the subspace of divergence-free functions.
		\begin{lemma}\label{lem:consistency}
			Let $(\vec{u}, p)$ be the solution of the Stokes problem with unit viscosity.
			The consistency error estimate
			\begin{align}
				\abs{a_h(\vec{u}, \vec{v}_h) - (\vec{f}, \vec{v}_h)} &\lesssim h \norm{\vec{v}_h}_{1,h} \norm{\vec{f}}_0 \qquad \forall \vec{v}_h\in \vec{X}_h^0\label{eq:consistency2}
			\end{align}
			holds, where the constant is independent of the aspect ratio of the mesh and the mesh size parameter $h$.
		\end{lemma}
		\begin{proof}
			We write for $\vec{v}_h\in\vec{X}_h^0$
			\begin{align*}
				\abs{a_h(\vec{u},\vec{v}_h) - (\vec{f},\vec{v}_h)} \leq \abs{a_h(\vec{u},\vec{v}_h) + b_h(\vec{v}_h,p) - (\vec{f},\vec{v}_h)} + \abs{b_h(\vec{v}_h,p)},
			\end{align*}
			where the first term vanishes since $\vec{X}_h^0 \subset \vec{X}_h \subset \vec{X}$.
			Estimating the second term, using the $L^2$-projection operator $\pi_h$ onto $Q_h$, we get
			\begin{align*}
				\abs{a_h(\vec{u},\vec{v}_h) - (\vec{f},\vec{v}_h)} &\leq \abs{b_h(\vec{v}_h,p)} = \abs{b_h(\vec{v}_h,\pi_h p) + b_h(\vec{v}_h,p - \pi_h p)} = \abs{b_h(\vec{v}_h,p - \pi_h p)} \\
				&\leq \norm{\nabla_h\cdot\vec{v}_h}_0 \norm{p-\pi_h p}_0 \leq \norm{\vec{v}_h}_{1,h}\norm{p-\pi_h p}_0.
			\end{align*}
			The error of the $L^2$-projection onto the piecewise constant functions can be estimated using \cite[Theorem 1.103]{ErnGuermond2004} which, using the result that the Stokes solution is bounded by the data function, see, e.g., \cite[Theorem 4.3]{ErnGuermond2004}, leads to the final estimate
			\begin{align*}
				\abs{a_h(\vec{u},\vec{v}_h) - (\vec{f},\vec{v}_h)} &\lesssim h \norm{\vec{v}_h}_{1,h}\norm{p}_1 \lesssim h \norm{\vec{v}_h}_{1,h}\norm{\vec{f}}_0. \qedhere
			\end{align*}
		\end{proof}
		
		\begin{lemma}\label{lem:approximation}
			Let $(\vec{u}, p)$ be the solution of the Stokes problem \eqref{eq:Stokesweak}.
			Then for the Bernardi--Raugel element the approximation properties
			\begin{align*}
				\inf_{\vec{v}_h\in\vec{X}_h^0}\norm{\vec{u}-\vec{v}_h}_{1,h} &\lesssim h \norm{\mathbb{P}(\Delta\vec{u})}_0,&	\inf_{q_h\in Q_h}\norm{p-q_h}_0 &\lesssim h \norm{\vec{f}}_0
			\end{align*}
			hold, where the constants are independent of the aspect ratio of the mesh and the mesh size parameter $h$.
		\end{lemma}
		\begin{proof}
			We first need the stability estimate for the Bernardi--Raugel interpolation operator from \cite[Section 5.2]{ApelNicaise2004}, where it was shown that for $\vec{v} \in \vec{H}^2(\Omega)$ the estimate
			\begin{equation}\label{eq:BR_stability}
				\norm{\BRI \vec{v}}_{1,h} \lesssim \norm{\vec{v}}_{1,h} + h \abs{\vec{v}}_2
			\end{equation}
			holds on the types of meshes we use.
			With the technique from the proof of \cite[II.(1.16)]{GiraultRaviart1986}, we get
			\begin{equation*}
				\inf_{\vec{v}_h\in\vec{X}_h^0} \norm{\vec{u}-\vec{v}_h}_{1,h} \lesssim \inf_{\vec{v}_h\in\vec{X}_h} \norm{\vec{u}-\vec{v}_h}_{1,h} \lesssim \norm{\vec{u}-\BRI\vec{u}}_{1,h},
			\end{equation*}
			so that now only the error of the Bernardi--Raugel interpolation needs to be estimated.
			Since the operator $\BRI$ preserves linear polynomials we can use the stability estimate \eqref{eq:BR_stability} and a Bramble--Hilbert type argument, which in the end leads to the estimate
			\begin{equation*}
				\inf_{\vec{v}_h\in\vec{X}_h^0} \norm{\vec{u}-\vec{v}_h}_{1,h} \lesssim h \abs{\vec{u}}_2.
			\end{equation*}
			As $\vec{u}\in\vec{H}^2(\Omega)$ is the Stokes velocity solution for data $\vec{f}$, we know, see, e.g., \cite[Lemma 2]{ApelKempf2021}, that it also solves a Stokes system with data $\nu^{-1}\mathbb{P}\vec{f}$.
			We can thus again use that the Stokes solution is bounded by the data function, see, e.g., \cite[Theorem 4.3]{ErnGuermond2004}, and estimate
			\begin{equation*}
				\abs{\vec{u}}_2 \lesssim \nu^{-1}\norm{\mathbb{P}\vec{f}}_0.
			\end{equation*}
			With \cite[Equation (9)]{ApelKempf2021} we now get the desired estimate
			\begin{equation*}
				\inf_{\vec{v}_h\in\vec{X}_h^0} \norm{\vec{u}-\vec{v}_h}_{1,h} \lesssim h \norm{\mathbb{P}(\Delta \vec{u})}_0.
			\end{equation*}
			The estimate for the pressure can be acquired by again using the error estimate for the $L^2$-projection into piecewise constants $\pi_h$ from \cite[Theorem 1.103]{ErnGuermond2004}, with which we can compute
			\begin{equation*}
			\inf_{q_h\in Q_h} \norm{p-q_h}_0 = \norm{p-\pi_h p}_0 \lesssim h \norm{p}_1 \lesssim h \norm{\vec{f}}_0.\qedhere
			\end{equation*} 
		\end{proof}

		With these lemmas as preparation, we are able to prove the discretization error estimates.
		\begin{theorem}\label{th:main_result}
			Let $\Omega \subset \mathbb{R}^2$ be a convex polygon, $\vec{X}_h$, $Q_h$ the Bernardi--Raugel finite element pair and $(\vec{u},p)$, $(\vec{u}_h,p_h)$ the solutions to \eqref{eq:Stokesweak} and \eqref{eq:Stokesdiscrete}.
			Further let the reconstruction operator $I_h$ be defined by either $(I_h\vec{v}_h)|_T = \BDMIk{1}\vec{v}_h|_T$ or $(I_h\vec{v}_h)|_T = \RTIk{0}\vec{v}_h|_T$ for all $\vec{v}_h\in\vec{X}_h$ and $T\in\Mesh$, and let $\Mesh$ satisfy the mesh conditions from \cite{ApelNicaise2004}.
			Then the estimates
			\begin{align*}\label{eq:main_result}
				\norm{\vec{u}-\vec{u}_h}_{1,h} &\lesssim \inf_{\vec{v}_h\in\vec{X}_h^0} \norm{\vec{u}-\vec{v}_h}_{1,h} + h \norm{\mathbb{P}( \Delta \vec{u})}_0,\\
				\norm{p-p_h}_0 &\lesssim \inf_{q_h\in Q_h} \norm{p-q_h}_0 + \frac{\nu}{\tilde{\beta}} \inf_{\vec{v}_h\in\vec{X}_h^0}\norm{\vec{u}-\vec{v}_h}_{1,h} + \frac{h}{\tilde{\beta}} \norm{\vec{f}}_0
			\end{align*}
			hold, where $\tilde{\beta}$ is the discrete inf-sup constant.
		\end{theorem}
		\begin{proof}
			Let $\vec{v}_h\in \vec{X}_h^0$ be the best-approximation of $\vec{u}$ with respect to $\norm{\cdot}_{1,h}$ and set $\vec{w}_h=\vec{u}_h-\vec{v}_h\in\vec{X}_h^0$.
			Then due to the Pythagorean theorem we have
			\begin{equation}\label{eq:main_result_pythagoras}
				\norm{\vec{u}-\vec{u}_h}_{1,h}^2 = \norm{\vec{u}-\vec{v}_h}_{1,h}^2 + \norm{\vec{w}_h}_{1,h}^2.
			\end{equation}
			Using \eqref{eq:stokesdiscrete_elliptic} and $a_h(\vec{u}-\vec{v}_h,\vec{w}_h)=0$ we can estimate
			\begin{align*}
				\norm{\vec{w}_h}_{1,h}^2 &= a_h(\vec{w}_h, \vec{w}_h) = a_h(\vec{u}_h-\vec{v}_h,\vec{w}_h)=a_h(\vec{u}-\vec{v}_h, \vec{w}_h) - a_h(\vec{u},\vec{w}_h) + a_h(\vec{u}_h,\vec{w}_h)\\
				&\leq \abs{a_h(\vec{u},\vec{w}_h) - \nu^{-1}(\vec{f},I_h\vec{w}_h)}.
			\end{align*}
			Dividing by $\norm{\vec{w}_h}_{1,h}$ and combining this inequality with \eqref{eq:main_result_pythagoras} yields
			\begin{equation}\label{eq:main_result_proof_0}
				\norm{\vec{u}-\vec{u}_h}_{1,h} \leq \norm{\vec{u}-\vec{v}_h}_{1,h} +  \frac{ \abs{a_h(\vec{u}, \vec{w}_h) - \nu^{-1} (\vec{f},I_h\vec{w}_h)}}{\norm{\vec{w}_h}_{1,h}}.
			\end{equation}
			Recall the Helmholtz--Hodge decomposition of the data $\vec{f} = \mathbb{P} \vec{f} + \nabla \phi$ and note that $\nabla\cdot I_h\vec{w}_h = 0$ due to \Cref{lem:reconstruction_operator} and $\vec{w}_h\in\vec{X}_h^0$.
			With $(\nabla\phi,I_h\vec{w}_h) = 0$ we get
			\begin{align}
				\abs{a_h(\vec{u},\vec{w}_h) - \frac{1}{\nu}(\vec{f},I_h\vec{w}_h)} &= \abs{a_h(\vec{u},\vec{w}_h) - \nu^{-1}(\mathbb{P}\vec{f},I_h\vec{w}_h)} \nonumber\\
				&= \abs{a_h(\vec{u},\vec{w}_h) - \nu^{-1}(\mathbb{P} \vec{f},\vec{w}_h) + \nu^{-1}(\mathbb{P} \vec{f},\vec{w}_h - I_h\vec{w}_h)} \nonumber\\
				&\leq  \abs{a_h(\vec{u},\vec{w}_h) - \nu^{-1}(\mathbb{P} \vec{f},\vec{w}_h)} + \abs{\nu^{-1}(\mathbb{P} \vec{f},\vec{w}_h - I_h\vec{w}_h)}. \label{eq:main_result_proof_1}
			\end{align}
			By \cite[Lemma 2]{ApelKempf2021}, $\vec{u}$ is also the velocity solution of the Stokes problem with unit viscosity and right hand side $\nu^{-1}\mathbb{P}\vec{f}$, which means that we can apply the consistency estimate of \Cref{lem:consistency}, which yields
			\begin{align}
				\abs{a_h(\vec{u},\vec{w}_h) - \nu^{-1}(\mathbb{P} \vec{f},\vec{w}_h)} &\lesssim \nu^{-1} h \norm{\vec{w}_h}_{1,h} \norm{\mathbb{P}\vec{f}}_0. \label{eq:main_result_proof_2}
			\end{align}
			The second term in \eqref{eq:main_result_proof_1} can be estimated using the Cauchy--Schwarz inequality and the interpolation error estimate for the reconstruction operator $I_h$ from \Cref{lem:reconstruction_operator}, which gets us
			\begin{equation}\label{eq:main_result_proof_3}
			\abs{\nu^{-1}(\mathbb{P} \vec{f},\vec{w}_h - I_h\vec{w}_h)} \leq \nu^{-1}\norm{\mathbb{P}\vec{f}}_0 \norm{\vec{w}_h - I_h\vec{w}_h}_0 \lesssim \nu^{-1} h \norm{\mathbb{P}\vec{f}}_0 \norm{\vec{w}_h}_{1,h}.
			\end{equation}
			We can now combine the individual estimates \eqref{eq:main_result_proof_2}, \eqref{eq:main_result_proof_3} with \eqref{eq:main_result_proof_1} and insert the result in \eqref{eq:main_result_proof_0}.
			Since $\vec{v}_h$ was chosen as the best-approximation of $\vec{u}$ in $\vec{X}_h^0$, we now have the final estimate
			\begin{equation*}
			\norm{\vec{u}-\vec{u}_h}_{1,h} \lesssim \inf_{\vec{v}_h\in\vec{X}_h^0}\norm{\vec{u}-\vec{v}_h}_{1,h} + h \norm{\mathbb{P}(\Delta \vec{u})}_0,
			\end{equation*}
			where we also used the identity \cite[Equation (9)]{ApelKempf2021}.
			
			To get the pressure estimate we also use the Pythagorean theorem to get
			\begin{equation*}
				\norm{p-p_h}_0^2 = \norm{p-\pi_h p}_0^2 + \norm{\pi_h p - p_h}_0^2,
			\end{equation*}
			where $\pi_h: L_0^2(\Omega) \to Q_h$ is the $L^2$-projection into the discrete pressure space.
			For the first term it holds $\norm{p-\pi_h p}_0^2 = \inf_{q_h\in Q_h}\norm{p-q_h}_0^2$.
			Since $\pi_h p-p_h \in Q_h$ and using the discrete inf-sup condition \eqref{eq:discrete_inf_sup} we get
			\begin{align}
				\norm{\pi_h p - p_h}_0 &\leq \frac{1}{\tilde{\beta}} \sup_{\vec{v}_h\in \vec{X}_h} \frac{b_h(\vec{v}_h,\pi_h p - p_h)}{\norm{\vec{v}_h}_{1,h}} = \frac{1}{\tilde{\beta}} \sup_{\vec{v}_h\in \vec{X}_h} \frac{b_h(\vec{v}_h,\pi_h p - p) + b_h(\vec{v}_h,p-p_h)}{\norm{\vec{v}_h}_{1,h}}.\label{eq:proof_p_estimate_1}
			\end{align}
			The first term in the numerator can be estimated using the Cauchy--Schwarz inequality, the error estimate for the $L^2$-projection into piecewise constant functions from \cite[Theorem 1.103]{ErnGuermond2004} which yields
			\begin{align}
				\abs{b_h(\vec{v}_h,\pi_hp - p)} \leq \norm{\nabla_h\cdot\vec{v}_h}_0 \norm{\pi_h p -p}_0 \lesssim \norm{\vec{v}_h}_{1,h} \norm{\pi_h p -p}_0 \lesssim h \norm{\vec{v}_h}_{1,h} \norm{\vec{f}}_0.\label{eq:proof_p_estimate_2}
			\end{align}
			Since $p_h$ solves the discrete problem we get for the second term
			\begin{align}
				\abs{b_h(\vec{v}_h,p-p_h)} &= \abs{b_h(\vec{v}_h,p) + \nu a_h(\vec{u}_h, \vec{v}_h) - (\vec{f}, I_h\vec{v}_h)} \nonumber\\
				&= \abs{\nu a_h(\vec{u},\vec{v}_h) + b_h(\vec{v}_h, p) - (\vec{f}, \vec{v}_h) + \nu a_h(\vec{u}_h-\vec{u},\vec{v}_h) + (\vec{f}, \vec{v}_h - I_h\vec{v}_h)} \nonumber\\
				&\lesssim \nu \norm{\vec{u}-\vec{u}_h}_{1,h}\norm{\vec{v}_h}_{1,h} + h \norm{\vec{f}}_0\norm{\vec{v}_h}_{1,h}, \label{eq:proof_p_estimate_3}
			\end{align}
			where in the last step the consistency of the method, the Cauchy--Schwarz inequality and the interpolation error estimate from \Cref{lem:reconstruction_operator} was used.
			Now putting \eqref{eq:proof_p_estimate_2} and \eqref{eq:proof_p_estimate_3} into \eqref{eq:proof_p_estimate_1} and using the estimate for the velocity error yields the claimed pressure estimate.
		\end{proof}
	
		\begin{corollary}
			Under the assumptions from \Cref{th:main_result} we have the estimates
			\begin{align*}
				\norm{\vec{u}-\vec{u}_h}_{1,h} &\lesssim h \norm{\mathbb{P}(\Delta \vec{u})}_0, & \norm{p-p_h}_0 &\lesssim h \tilde{\beta}^{-1} \norm{\vec{f}}_0,
			\end{align*}
		\end{corollary}
		\begin{proof}
			This is a direct application of \Cref{lem:approximation} to the estimates from \Cref{th:main_result}.
		\end{proof}
			
	\section{Numerical example}
		We now present an academic numerical example to see the performance of the method on anisotropic meshes.
		The example employs a manufactured solution of the Stokes equations on the unit square $\Omega = (0,1)^2$ described by the velocity and pressure functions
		\begin{equation*}
			\vec{u}(\vec{x}) = \begin{pmatrix}\tanh\left(\frac{y}{\sqrt{\varepsilon}}\right) \\ 0 \end{pmatrix}, \qquad p(\vec{x}) = \tanh\left(\frac{y}{\sqrt{\varepsilon}}\right) - C(\varepsilon),
		\end{equation*}
		with a positive parameter $\varepsilon$.
		Both functions exhibit a boundary layer near $y = 0$, as can be seen in the visualization in \Cref{fig:boundary_layer}.
		\begin{figure}[t]
			\centering
			\includegraphics{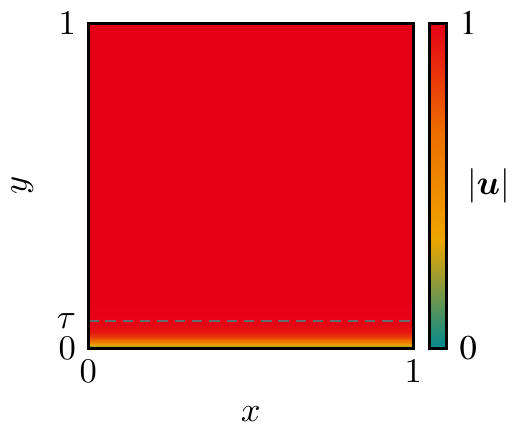}\hfill
			\includegraphics{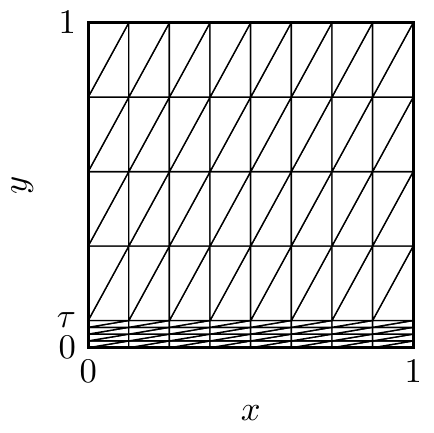}
			\caption{Left: Magnitude of velocity solution for $\varepsilon=10^{-3}$. Right: Shishkin-type mesh.}
			\label{fig:boundary_layer}
		\end{figure}
		\begin{figure}[t]
			\centering
			\includegraphics{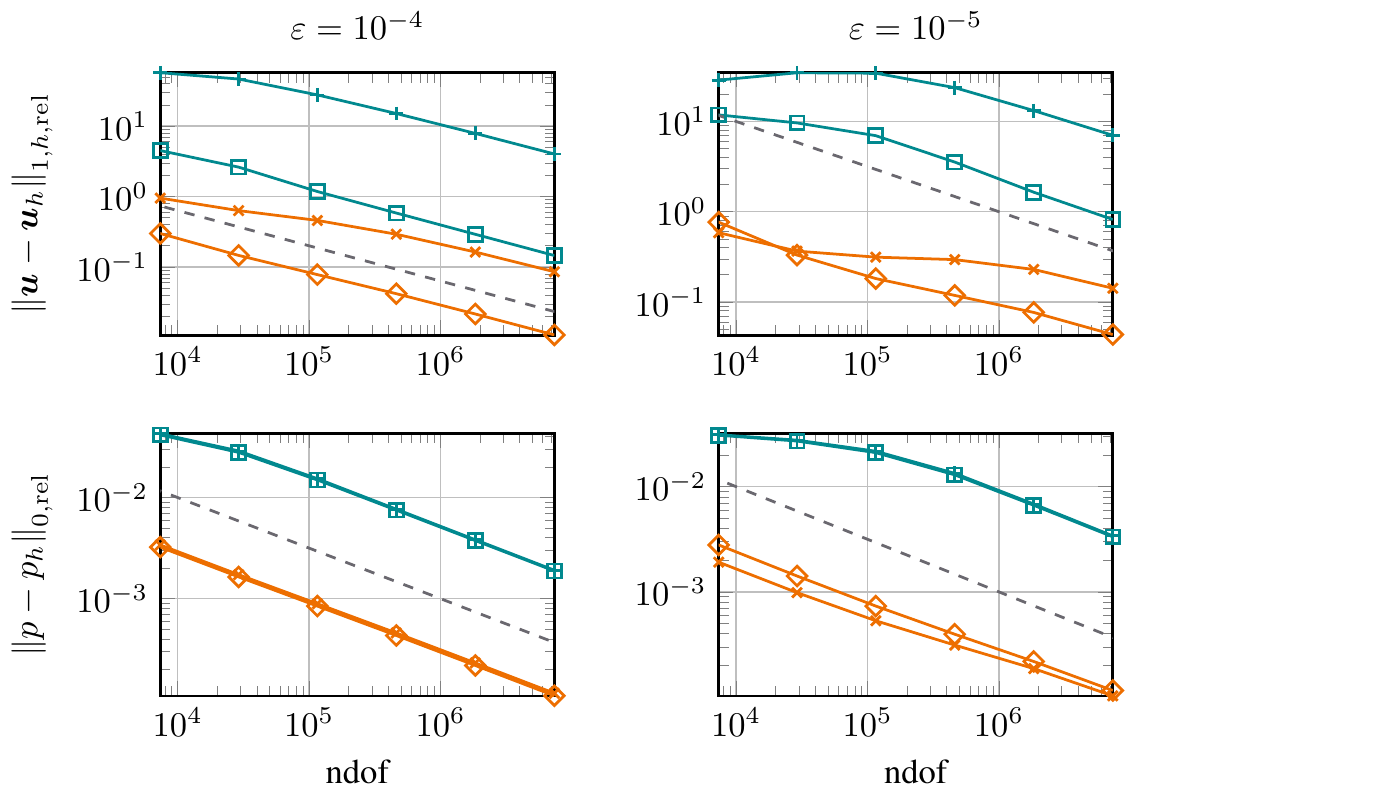}

			\includegraphics{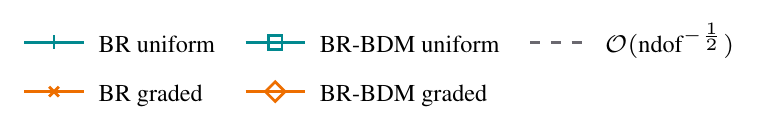}
			\caption{Convergence plots for the boundary layer example for $\varepsilon\in\{10^{-4},10^{-5}\}$, $\nu=10^{-4}$, with BR and BR-BDM methods.\\}
			\label{fig:boundary_layer_convergence_BR}
		\end{figure}
	
		The functions can be viewed as a fluid flow along a wall with no-slip boundary condition.
		The parameter $\varepsilon$ can be used to adjust the width of the boundary layer.
		Defining the boundary layer width as the distance from the wall where $99 \%$ of the free flow velocity is reached, we compute
		\begin{align*}
			\abs{\vec{u}(\cdot,\tau)} &= \tanh\left(\frac{\tau}{\sqrt{\varepsilon}}\right) = 0.99 &	\Leftrightarrow \qquad\qquad \tau &= 0.5\sqrt{\varepsilon}\ln\left(199\right)\approx 2.65\sqrt{\varepsilon}
		\end{align*}
		for the transition point parameter $\tau$ of the Shishkin-type meshes we want to use. 
		This type of mesh has a uniform element size in $x$-direction and half of the total elements up to $\tau$ in the $y$-direction, see the bottom illustration in \Cref{fig:boundary_layer}.
		The constant $C(\varepsilon)$ is needed to set the mean pressure to zero and can be computed by 
		\begin{equation*}
			C(\varepsilon) = \int_\Omega \tanh\left(\frac{y}{\sqrt{\varepsilon}}\right) \dd{\vec{x}} = \sqrt{\varepsilon} \ln (\cosh(\varepsilon^{-\frac{1}{2}})).
		\end{equation*}
	
		Computations were performed with the BR and BR-BDM methods for parameter choices $\varepsilon\in\{10^{-4},10^{-5}\}$ and $\nu=10^{-4}$ on uniform and Shishkin-type meshes.
		For the presentation of the numerical results we use the relative errors
		\begin{align*}
			&\norm{\vec{u}-\vec{u}_h}_{1,h,\mathrm{rel}} = \frac{\norm{\vec{u}-\vec{u}_h}_{1,h}}{\norm{\vec{u}}_{1,h}}, && \norm{p-p_h}_{0,\mathrm{rel}} = \frac{\norm{p-p_h}_{0}}{\norm{p}_0}.
		\end{align*}
		
		The results are shown in \Cref{fig:boundary_layer_convergence_BR}.
		The plots show on the one hand the clear advantage of the pressure-robust methods, where the velocity errors are significantly smaller than for the standard method.
		On the other hand, the effect of the anisotropic mesh grading is obvious in the velocity errors as well as the pressure errors.
	
	\printbibliography
	
\end{document}